\documentclass[a4paper,12pt]{amsart}
\usepackage{amsmath,amssymb,amsxtra,comment,graphicx,psfrag}
\usepackage{bm,mathrsfs}
\usepackage{mathtools}
\usepackage{xspace}
\usepackage{color}
\usepackage{xcolor}
\usepackage{hhline}

\definecolor{forestgreen}{rgb}{0.13, 0.55, 0.13}
\definecolor{lightblue}{rgb}{0.68, 0.85, 0.9}
\usepackage[colorlinks=true,
            linkcolor=blue,
            urlcolor  = blue,
            citecolor=forestgreen,
            anchorcolor = lightblue]{hyperref}

\usepackage{enumerate}

\usepackage{hyperref}
\usepackage{pdfsync}
\usepackage{hhline}
\usepackage{fancyhdr}
\usepackage{epsfig}
\usepackage{epsfig,subfigure,epstopdf}
\allowdisplaybreaks

\usepackage{pstricks,pst-plot,pst-func}
\usepackage{pspicture}
\usepackage{curves}

\include{rgb}

\def\e{{\rm e}}
\def\i{{\rm i}}
\def\C{{\mathbb C}}
\def\N{{\mathbb N}}

\def\Re {{\mathbb R}}

\def\K{{\mathscr{K}}}

\DeclareMathOperator\Real {Re}

\definecolor{mygray}{rgb}{0.9,0.9,0.9}

\begin{document}
\theoremstyle{plain}
\newtheorem{theorem}{Theorem}[section]
\newtheorem{lemma}{Lemma}[section]
\newtheorem{proposition}{Proposition}[section]
\newtheorem{corollary}{Corollary}[section]

\theoremstyle{definition}
\newtheorem{definition}[corollary]{Definition}

\newtheorem{example}{Example}[section]

\newtheorem{remark}{Remark}[section]
\newtheorem{remarks}[remark]{Remarks}
\newtheorem{note}{Note}
\newtheorem{case}{Case}

\numberwithin{equation}{section}
\numberwithin{table}{section}
\numberwithin{figure}{section}

\title[Six-step BDF method]
{The energy technique for the six-step BDF method}
\author[Georgios Akrivis]{Georgios Akrivis}
\address{Department of Computer Science and Engineering, University of Ioannina, 451$\,$10 Ioannina, Greece,
and Institute of Applied and Computational Mathematics, FORTH, 700$\,$13 Heraklion, Crete, Greece}
\email {\href{mailto:akrivis@cse.uoi.gr}{akrivis{\it @\,}cse.uoi.gr}}

\author[Minghua Chen]{Minghua Chen}
\address{School of Mathematics and Statistics, Gansu Key Laboratory of Applied Mathematics and Complex Systems,
 Lanzhou University, Lanzhou 730000, P.R. China}
\email {\href{mailto:chenmh@lzu.edu.cn}{chenmh{\it @\,}lzu.edu.cn}}

\author[Fan Yu]{Fan Yu}
\address{School of Mathematics and Statistics, Gansu Key Laboratory of Applied Mathematics and Complex Systems,
 Lanzhou University, Lanzhou 730000, P.R. China}
\email {\href{mailto:yuf17@lzu.edu.cn}{yuf17{\it @\,}lzu.edu.cn}}

\author[Zhi Zhou]{Zhi Zhou}
\address{Department of Applied Mathematics, The Hong Kong Polytechnic University, Kowloon, Hong Kong, P.R. China}
\email {\href{mailto:zhizhou@polyu.edu.hk}{zhizhou{\it @\,}polyu.edu.hk}}

\thanks{
The research of M.\ Chen and F.\ Yu is partially supported by NSFC 11601206,
and the research of Z.\ Zhou is supported by Hong Kong RGC grant (No. 25300818).
}

\date{\today}

\keywords{Six-step BDF method, multipliers, parabolic equations, stability estimate, energy technique}
\subjclass[2010]{Primary 65M12, 65M60; Secondary 65L06.}

\begin{abstract}
In combination with the Grenander--Szeg\"o theorem, we observe that a relaxed
positivity condition on multipliers, milder than the   basic   
requirement of the
Nevanlinna--Odeh multipliers that the sum of the absolute values of their components
is strictly less than $1$,  makes the energy technique applicable to the stability
analysis of BDF methods for parabolic equations with selfadjoint elliptic part.
This is particularly useful for the six-step BDF method
for which no Nevanlinna--Odeh multiplier exists.
We introduce multipliers satisfying the positivity property for the six-step BDF method
and establish stability of the method for parabolic equations. 
\end{abstract}

\maketitle


\section{Introduction}\label{Se:intro}
Let $T >0, u^0\in H,$ and consider the initial value
problem of seeking $u \in C((0,T];D(A))\cap C([0,T];H)$ satisfying
\begin{equation}
\label{ivp}
\left \{
\begin{aligned}
&u' (t) + Au(t) = 0, \quad 0<t<T,\\
& u(0)=u^0 ,
\end{aligned}
\right .
\end{equation}
with  $A$ a positive definite, selfadjoint, linear operator on a
Hilbert space $(H, (\cdot , \cdot )) $ with domain  $D(A)$
dense in $H.$ 

We consider the $q$-step backward difference formula (BDF) method, generated by the polynomials  $\alpha$ and $\beta,$
\begin{equation}
\label{BDF1}
\alpha (\zeta)= \sum_{j=1}^q \frac 1j \zeta^{q-j} (\zeta-1)^j
=\sum_{j=0}^q \alpha_j \zeta^j,
\quad \beta(\zeta)=\zeta^q.
\end{equation}
The BDF methods are $A(\vartheta_q)$-stable with
$\vartheta_1=\vartheta_2=90^\circ,
\vartheta_3\approx 86.03^\circ, \vartheta_4\approx 73.35^\circ, \vartheta_5\approx 51.84^\circ$ and $\vartheta_6 \approx 17.84^\circ$;
see \cite[Section V.2]{HW}. Exact values of $\vartheta_q, q=3,4,5,6,$
are given in \cite{AK2}. The order of the $q$-step method is $q.$

Let $N\in \N,$ $\tau:=T/N$ be the time step, and $t^n :=n \tau,
n=0,\dotsc ,N,$ be a uniform partition of the interval $[0,T].$
We recursively define a sequence of approximations $u^m$ to
the nodal values $ u(t^m)$ by the $q$-step BDF method,
\begin{equation}
\label{ab}
\sum_{i=0}^q \alpha_i u^{n+i}+\tau A u^{n+q}=0,\quad n=0,\dotsc,N-q,
\end{equation}
assuming that starting approximations $u^0, \dotsc, u^{q-1}$ are given.

Let $| \cdot |$  denote the norm on $H$ induced by the inner product $(\cdot , \cdot )$, and introduce on $V, V:=D(A^{1/2}),$
the norm $\| \cdot \|$   by $\| v\| :=| A^{1/2} v |.$
We identify $H$ with its dual, and denote by $V'$ the dual of $V$,
and by $\| \cdot \|_\star$ the dual norm on $V', \|v \|_\star=| A^{-1/2} v |.$
We shall use the notation $(\cdot , \cdot )$ also for the antiduality
pairing between $V'$ and $V.$

Stability of the A-stable one- and two-step BDF methods  \eqref{ab}
can be easily established by the energy method. The powerful
 Nevanlinna--Odeh multiplier technique extends the applicability
 of the energy method to the non A-stable three-, four- and five-step BDF methods.
In contrast, as we shall see, no Nevanlinna--Odeh multiplier exists for the six-step BDF method.
Here, we show that, in combination with the Grenander--Szeg\"o theorem,
the energy technique is applicable even with multipliers satisfying milder
requirements than Nevanlinna--Odeh multipliers. We introduce such
multipliers for the six-step BDF method and prove stability by the
energy technique.

An outline of the paper is as follows: In Section \ref{Se:mult},  we relax the requirements
on the multipliers for BDF methods and present multipliers for the six-step
BDF method. In Section \ref{Se:stab}, we use a new multiplier in combination with
the Grenander--Szeg\"o theorem and prove stability of the six-step BDF method
for the initial value problem \eqref{ivp}.

\section{Multipliers for the six-step BDF method}\label{Se:mult}

Multipliers for the three-, four- and five-step BDF methods were introduced
by Nevanlinna and Odeh already in 1981, see \cite{NO}, to make the energy
method applicable to the stability analysis of these methods for parabolic
equations; no multipliers are required for the A-stable one- and two-step BDF methods.
The multiplier technique became widely known and popular after its first actual application
to the stability analysis  for parabolic equations by  Lubich, Mansour, and Venkataraman
in 2013; see \cite{LMV}.

The multiplier technique hinges on the celebrated equivalence of A- and G-stability for multistep methods
by Dahlquist; see \cite{D}.

\begin{lemma}[\cite{D}; see also \cite{BC} and
{\cite[Section V.6]{HW}}]\label{Le:Dahl}
Let $\alpha(\zeta)=\alpha_q\zeta^q+\dotsb+\alpha_0$ and
$\mu(\zeta)=\mu_q\zeta^q+\dotsb+\mu_0$ be polynomials,
with real coefficients, of degree
at most $q\ ($and at least one of them of degree $q)$
that have no common divisor.
Let $(\cdot,\cdot)$ be a real inner product with associated norm $|\cdot|.$
If
\begin{equation}
\label{A}
\Real \frac {\alpha(\zeta)}{\mu(\zeta)}>0\quad\text{for }\, |\zeta|>1,
\tag{A}
\end{equation}
then there exists a positive definite symmetric matrix $G=(g_{ij})\in \Re^{q,q}$
and real $\delta_0,\dotsc,\delta_q$ such that for $v^0,\dotsc,v^{q}$ in the inner product space,
\begin{equation}
\label{G}
 \Big (\sum_{i=0}^q\alpha_iv^{i},\sum_{j=0}^q\mu_jv^{j}\Big )=
\sum_{i,j=1}^qg_{ij}(v^{i},v^{j})
-\sum_{i,j=1}^qg_{ij}(v^{i-1},v^{j-1})
+\Big |\sum_{i=0}^q\delta_iv^{i}\Big |^2.    
\tag{G}
\end{equation}
\end{lemma}

\begin{definition}[Multipliers and Nevanlinna--Odeh multipliers]\label{De:mult}
Let $\alpha$ be the generating polynomial of the $q$-step BDF method defined in \eqref{BDF1}.
Consider a $q$-tuple $(\mu_1,\dotsc,\mu_q)$ of real numbers such that
with the given $\alpha$ and
$\mu(\zeta):=\zeta^q-\mu_1\zeta^{q-1}-\dotsb-\mu_q,$
the pair $(\alpha,\mu)$ satisfies the A-stability condition \eqref{A},
and, in addition, the polynomials $\alpha$ and $\mu$ have no common divisor.
Then, we call $(\mu_1,\dotsc,\mu_q)$  \emph{Nevanlinna--Odeh multiplier} for the
$q$-step BDF method if
\begin{equation}
\label{NO-multiplier}
1-|\mu_1|-\dotsb-|\mu_q|>0, \tag{P1}
\end{equation}
and simply \emph{multiplier} if it satisfies the \emph{positivity} property
\begin{equation}
\label{pos-prop}
1-\mu_1\cos x-\dotsb-\mu_q\cos (qx) >0 \quad \forall x \in \Re. \tag{P2}
\end{equation}
\end{definition}


Notice that, with the notation of this definition, \eqref{A} and \eqref{G}, respectively,
mean that the $q$-step scheme described by the parameters
$\alpha_q,\dotsc,\alpha_0,1,-\mu_1,\dotsc,-\mu_q$
and the corresponding one-leg method are A- and G-stable, respectively.
Of course, these are necessarily low order methods but this is irrelevant here;
we do not compute with them; we only use them to establish stability of
the $q$-step BDF method.

Optimal Nevanlinna--Odeh multipliers, i.e., the ones with minimal $|\mu_1|+\dotsb+|\mu_q|$,
for the three-, four- and five-step BDF methods were given in \cite{AK1}.

Some comments on the requirements in Definition \ref{De:mult} and
their role in the stability analysis are in order.
To prove stability of the method by the energy technique, we test \eqref{ab}
by $u^{n+q}-\mu_1u^{n+q-1}-\dotsb-\mu_qu^n$ and obtain
\begin{equation}
\label{ab-energy}
\Big (\sum_{i=0}^q \alpha_i u^{n+i},u^{n+q}-\sum_{j=1}^q \mu_ju^{n+q-j}\Big )
+\tau \Big (A u^{n+q},u^{n+q}-\sum_{j=1}^q \mu_j u^{n+q-j}\Big ) =0,
\end{equation}
$n=0,\dotsc,N-q.$ The first term on the left-hand side can be estimated
from below using \eqref{G}; this is the motivation for the requirement \eqref{A}.
Which one of the other two conditions, \eqref{NO-multiplier} or \eqref{pos-prop},
enters into the stability analysis, depends on the way we handle the second term
on the left-hand side of \eqref{ab-energy}.
If we estimate this term from below at every time level and then sum over $n$,
requirement \eqref{NO-multiplier} is crucial; cf., e.g., \cite{AL}, \cite{A2}, \cite{AK1}.
Instead, if we sum over $n$ and subsequently estimate the sum of the second
terms, the relaxed positivity condition \eqref{pos-prop} suffices.
In the latter approach, in view of the Grenander--Szeg\"o theorem,
\eqref{pos-prop} ensures that symmetric band Toeplitz matrices, of any dimension,
with generating function the positive trigonometric polynomial
$(1-\varepsilon)-\mu_1\cos x-\dotsb-\mu_q\cos (qx),$
for sufficiently small $\varepsilon$, are positive definite;
see section \ref{Se:stab}.

It is well known that any multiplier for the $q$-step BDF method  satisfies the
property
\begin{equation}
\label{multiplier-lower-b}
|\mu_1|+\dotsb+|\mu_q|\geqslant \cos\vartheta_q;
\end{equation}
see \cite{NO}. In particular, for the six-step BDF method this means that $|\mu_1|+\dotsb+|\mu_6|\geqslant 0.9516169.$
Actually, as we shall see, no Nevanlinna--Odeh multiplier exists for the six-step BDF method; see Remark \ref{rem:mult-general}.
This was the motivation for our relaxation on the requirements for multipliers.
Fortunately, the relaxed positive condition \eqref{pos-prop} leads to a positive result.

\begin{proposition}[A multiplier for the six-step BDF method]\label{Pr:mult-six}
The set of numbers
\begin{equation}
\label{mu2}
\mu_1=\frac {13}{9},\quad \mu_2=-\frac {25}{36},\quad \mu_3=\frac 19,\quad \mu_4=\mu_5=\mu_6=0,
\end{equation}
is a multiplier for the six-step BDF method.
\end{proposition}

\begin{proof}
The proof consists of two parts; we first prove the A-stability property
 \eqref{A} and subsequently the positivity property \eqref{pos-prop}.

\emph{A-stability property \eqref{A}}.
The corresponding polynomial $\mu$ is
\begin{equation}
\label{mu1}
\mu(\zeta)=\zeta^3\big (\zeta-\frac 12\big )^2\big (\zeta-\frac 49\big )
= \zeta^6-\frac{13}{9}\zeta^5+\frac{25}{36}\zeta^4-\frac 19\zeta^3
=\frac 1{36}\zeta^3(36\zeta^3-52\zeta^2+25\zeta-4).
\end{equation}

We recall the generating polynomial $\alpha$ of the six-step BDF method,
\begin{equation*}
\label{alpha1}
60 \alpha(\zeta)=147\zeta^6-360\zeta^5+450\zeta^4-400\zeta^3+225\zeta^2- 72\zeta+10.
\end{equation*}
First, $\alpha(1/2)=-37/3840$ and $\alpha(4/9)=-0.003730423508913,$ whence
the polynomials $\alpha$ and $\mu$ have no common divisor.

Now, $\alpha(z)/\mu (z)$ is holomorphic outside the unit disk in the
complex plane, and
\[\lim_{|z|\to \infty}\frac {\alpha(z)}{\mu (z)}=\alpha_6=\frac {147}{60}>0.\]
Therefore, according to the maximum principle for harmonic functions,
the A-stability property \eqref{A} is equivalent to
\begin{equation*}
\Real \frac {\alpha(\zeta)}{\mu(\zeta)}\geqslant 0 \quad \forall \zeta\in \K,
\end{equation*}
with $\K$ the unit circle in the complex plane, $\K:=\{\zeta\in \C : |\zeta|=1\},$
i.e., equivalent to
\begin{equation}
\label{Real1}
\Real \big [\alpha(\e^{\i \varphi})\mu (\e^{-\i \varphi})\big ]\geqslant 0 \quad \forall \varphi \in \Re.
\end{equation}
%


In view of \eqref{mu1}, the desired property \eqref{Real1} takes the form
\begin{equation}
\label{Real2}
\Real \big [60\alpha(\e^{\i \varphi})
\e^{-\i 3\varphi} \big (36\e^{-\i 3\varphi}-52\e^{-\i 2\varphi}+25 \e^{-\i \varphi}-4\big )\big ]\geqslant 0 \quad \forall \varphi \in \Re.
\end{equation}

Now, it is easily seen that
\begin{equation*}
\begin{aligned}
60\alpha(\e^{\i \varphi})\e^{-\i 3\varphi}
&=\big [157 \cos(3\varphi)-432 \cos(2\varphi) +675 \cos\varphi-400\big ]\\
&+\i\big [137 \sin(3\varphi)-288 \sin(2\varphi) +225 \sin\varphi\big ].
\end{aligned}
\end{equation*}
With $x:=\cos\varphi,$ recalling the elementary trigonometric identities
\[\cos(2\varphi) =2x^2-1,\  \cos(3\varphi) =4x^3-3x,\  \sin(2\varphi) =2x\sin\varphi,\  \sin(3\varphi) =(4x^2-1)\sin\varphi,\]
%
we easily see that
\begin{equation}
\label{Real3}
60\alpha(\e^{\i \varphi})\e^{-\i 3\varphi}
=4(1-x)(8+59x- 157x^2)+\i4 (137x^2-144x+22)\sin\varphi.
\end{equation}
%
Notice that the factor $1-x$ in the real part of $\alpha(\e^{\i \varphi})\e^{-\i 3\varphi}$
is due to the fact that $\alpha(1)=0.$ Similarly,
\[\begin{aligned}
36\e^{-\i 3\varphi}-52\e^{-\i 2\varphi}+25 \e^{-\i \varphi}-4
&=\big [36 \cos(3\varphi)-52 \cos(2\varphi) +25 \cos\varphi-4\big ]\\
&-\i\big [36 \sin(3\varphi)-52 \sin(2\varphi) +25 \sin\varphi\big ]
\end{aligned}\]
and
\begin{equation}
\label{Real4}
\begin{aligned}
36\e^{-\i 3\varphi}-52\e^{-\i 2\varphi}+25 \e^{-\i \varphi}-4
&= (144x^3-104x^2-83x+48 )\\
&-\i (144x^2-104x-11 )\sin\varphi.
\end{aligned}
\end{equation}
In view of \eqref{Real3} and \eqref{Real4}, the desired property  \eqref{Real2}
can be written in the form
\begin{equation}
\label{Real5}
4(1-x)P(x)\geqslant 0 \quad \forall x \in [-1,1]
\end{equation}
with
\[\begin{aligned}
P(x)&:=(8+59x-157x^2)(144x^3-104x^2-83x+48)\\
&+(1+x)(137x^2-144x+22)(144x^2-104x-11),
\end{aligned}\]
i.e.,
\begin{equation}
\label{Real6}
P(x)=2(71+611x+1334x^2-5150x^3+4784x^4-1440x^5).
\end{equation}

It is now easy to see that $P$ is positive in the interval $[-1,1],$ and thus that  \eqref{Real2} is valid.
First, the quadratic polynomial $ 71+611x+ 1334x^2$ is positive for all real $x$,
since it does not have real roots. All other terms are positive for negative $x,$
whence $P(x)$ is positive for negative $x.$ Furthermore, for $0\leqslant x  \leqslant 1,$ we obviously have
$71+611x\geqslant  682x^2,$ and can estimate $P(x)$ from below as follows
\[\begin{aligned}
P(x)&\geqslant 2x^2 (2016-5150x+4784x^2-1440x^3)\\
&=2x^2 \big [(2016-5150x+3344x^2)+1440x^2(1-x) \big ].
\end{aligned}\]
Again, the quadratic polynomial  $2016-5150x+3344x^2$ is positive for all real $x$, and the positivity of $P(x)$ follows.

\emph{Positivity property \eqref{pos-prop}.}
To prove the desired positivity property \eqref{pos-prop} for the multiplier \eqref{mu2},
we consider the function $f$,
\begin{equation}
\label{f}
f(x):=\frac{31}{32} -\frac{13}{9}\cos x+\frac{25}{36}\cos(2x)-\frac{1}{9}\cos(3x), \quad x\in \Re.
\end{equation}
Now, elementary trigonometric identities lead to the following form of $f$
\[f(x)=-\frac{4}{9}\cos^3 x+\frac{25}{18}\cos^2x-\frac{10}{9}\cos x+\frac{79}{288}.\]
Hence, we consider the polynomial $p,$
\begin{equation}
\label{p}
p(x):=-\frac{4}{9}x^3+\frac{25}{18}x^2-\frac{10}{9}x+\frac{79}{288},\quad x\in [-1,1].
\end{equation}
It is easily seen that $p$ attains its minimum at $x^\star=(25-\sqrt{145})/24$ 
and
\[p(x^\star)=0.009321552602567 > 0.\]
Therefore, $f$ is indeed positive; in particular,
the desired positivity property \eqref{pos-prop} is satisfied. See also Figure \ref{Fig:f}.
\end{proof}

\begin{figure}[!ht]
\centering
\psset{yunit=0.7cm,xunit=0.7}
\begin{pspicture}(-3.9,-0.6)(4.6,4.05)
\psaxes[ticks=none,labels=none,linewidth=0.6pt]{->}(0,0)(-4,-0.6)(4.3,3.67)%
[$\!\!x$,0][$ $,90]
\psset{plotpoints=10000}
\psFourier[cosCoeff=1.9375 -1.44444444444 0.69444444444  -.1111111111111, sinCoeff=0,
linewidth=0.5pt,linecolor=blue]{-3.1415926}{3.1415926}
\uput[0](-0.44,3.88){\small $y$}
\uput[0](-2.37,1.88){\small $f$}
\uput[0](2.74,-0.22){\small $\pi$}
\uput[0](-3.9,-0.22){\small $-\pi$}
\pscircle*(0,1){0.038}
\pscircle*(0,2){0.038}
\pscircle*(0,3){0.038}
\uput[0](-0.62,1){\small $1$}
\uput[0](-0.62,2){\small $2$}
\uput[0](-0.62,3){\small $3$}
\uput[0](-0.78,-0.26){\small $O$}
\pscircle*(3.1415926,0){0.038}
\pscircle*(-3.1415926,0){0.038}
\end{pspicture}
\hspace*{0.9cm}
\psset{yunit=0.7cm,xunit=2.7}
\hspace*{-0.05cm}\begin{pspicture}(-1.18,-0.6)(1.40,4.05)
\psaxes[ticks=none,labels=none,linewidth=0.6pt]{->}(0,0)(-1.2,-0.6)(1.3,3.67)%
[$\!\!x$,0][$ $,90]
\psset{plotpoints=10000}
\psPolynomial[coeff=0.27430556 -1.11111111111 1.38888889 -0.444444444, linewidth=0.5pt,linecolor=blue]{-1}{1}
\uput[0](-0.77,1.08){\small $p$}
\uput[0](-0.17,3.88){\small $y$}
\uput[0](0.86,-0.27){\small $1$}
\uput[0](-1.2,-0.27){\small $-1$}
\pscircle*(0,1){0.038}
\pscircle*(0,2){0.038}
\pscircle*(0,3){0.038}
\uput[0](-0.22,1){\small $1$}
\uput[0](-0.22,2){\small $2$}
\uput[0](-0.22,3){\small $3$}
\uput[0](-0.28,-0.27){\small $O$}
\pscircle*(1,0){0.038}
\pscircle*(-1,0){0.038}
\end{pspicture}
\caption{The graphs of the function $f$ and the polynomial $p$ of \eqref{f} and \eqref{p}.}
\label{Fig:f}
\end{figure}
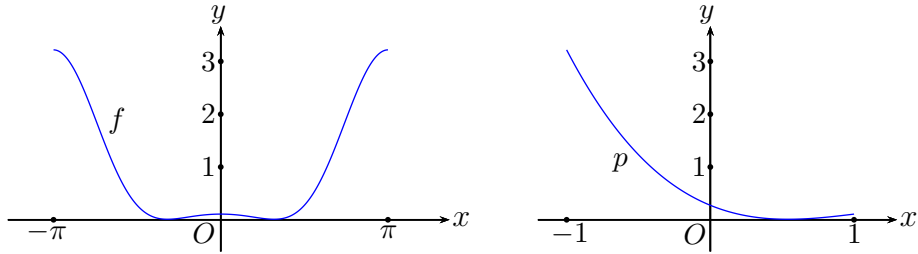

\subsection{On the conditions \eqref{pos-prop} and \eqref{NO-multiplier}}\label{SSe:discrepancy}
We briefly comment on the discrepancy between the conditions \eqref{pos-prop} and \eqref{NO-multiplier}.
Obviously, \eqref{NO-multiplier} implies \eqref{pos-prop}.

Let $S_q\subset \Re^q$ denote the region of the points $(\mu_1,\dotsc,\mu_q)$ satisfying the
positivity condition \eqref{pos-prop}. Since \eqref{NO-multiplier} and \eqref{pos-prop} are obviously equivalent
for $q$-tuples $(\mu_1,\dotsc,\mu_q)$ with only one nonvanishing component,
the intersection of $S_q$ with each coordinate axis is an interval of the form $(-1,1).$

Let us next focus on the instrumental case of the intersection of $S_q$ with the $\mu_1\mu_2$ plane, i.e.,
consider the set of points $(\mu_1,\dotsc,\mu_q)\in S_q$ with $\mu_3=\dotsb=\mu_q=0.$
Then, the positivity condition reads
\begin{equation}
\label{p-discrep1}
p(x):=1-\mu_1 x-\mu_2 (2x^2-1)>0,\quad x\in [-1,1].
\end{equation}
For $\mu_2=0,$ this condition is satisfied if and only if $|\mu_1|<1.$
For nonvanishing $\mu_2,$ the derivative of $p$ vanishes at
$x^\star=-\mu_1/(4\mu_2)$ and
\begin{equation}
\label{p-discrep2}
p(x^\star)=1+\mu_2 +\frac 18 \frac {\mu_1^2}  {\mu_2}.
\end{equation}
For positive $\mu_2$, this is a positive global maximum of $p.$
Therefore, in this case \eqref{p-discrep1} is satisfied if and only if $p(-1)$ and $p(1)$ are positive,
whence
\begin{equation}
\label{p-discrep3}
\mu_2 < 1-|\mu_1|.
\end{equation}
%
For negative $\mu_2$, the expression in  \eqref{p-discrep2} is a global minimum of $p.$
Now, we distinguish two subcases. It $|\mu_2|\leqslant |\mu_1|/4,$ then the minimum is attained
at a point $|x^\star|\geqslant 1,$ whence \eqref{p-discrep3} suffices for \eqref{p-discrep1}.
If, on the other hand, $|x^\star|< 1,$ then  \eqref{p-discrep1} is satisfied if and only
if the expression on the right-hand side of \eqref{p-discrep2} is positive, i.e.,
\[4\Big (\mu_2+\frac12 \Big )^2+\frac 12 \mu_1^2 <1;\]
that is,  $(\mu_1,\mu_2)$ lies in the interior of an ellipse.
Summarizing,  \eqref{p-discrep1}  is satisfied if and only if $(\mu_1,\mu_2)$ lie in the region
\[S:=\big \{(\mu_1,\mu_2):  -\frac {|\mu_1|}4\leqslant  \mu_2 < 1-|\mu_1|\big \}
\cup  \big \{(\mu_1,\mu_2):  4\Big (\mu_2+\frac12 \Big )^2+\frac 12 \mu_1^2 <1 \text{ and } |\mu_2|>\frac {|\mu_1|}4\big \}.\]
Notice that the lines $\mu_2=\pm(1- \mu_1)$ are tangent to the ellipse at their intersection points
with the lines  $\mu_2=\mp \mu_1/4$, respectively, i.e., at the points $(\pm 4/3,-1/3)$.
This is, of course, due to the fact that for these values the global minimum in \eqref{p-discrep2}
is attained at the points $x^\star=\pm 1.$ Therefore,  the intersection $S$ of $S_q$ with the
$\mu_1\mu_2$ plane is the union of two overlapping simple sets, a triangle and an ellipse,
\begin{equation}
\label{p-discrep4}
S=\big \{(\mu_1,\mu_2):  -\frac 13\leqslant  \mu_2 < 1-|\mu_1|\big \}
\cup  \big \{(\mu_1,\mu_2):  4\Big (\mu_2+\frac12 \Big )^2+\frac 12 \mu_1^2 <1 \big \};
\end{equation}
see Figure \ref{Fi:NO-posit}, right. Notice, in particular, that
\begin{equation}\label{eqn:mu1-mu2}
|\mu_1|<\sqrt{2} \quad \text{and}\quad |\mu_2|<1.
\end{equation}

Replacing $x$ by $x/2$ and by $x/3$, respectively, in the positivity condition \eqref{pos-prop},
it is obvious that the intersection of $S_q$ with the $\mu_2\mu_4$ plane, for $q\geqslant 4$,
and with the $\mu_3\mu_6$ plane, for $q=6$, respectively, is of the form \eqref{p-discrep4}
with $(\mu_1,\mu_2)$ replaced by $(\mu_2,\mu_4)$ and by $(\mu_3,\mu_6)$, respectively.

\begin{figure}[!ht]
\centering
\psset{yunit=1.7cm,xunit=1.7}
\begin{pspicture}(-1.6,-1.43)(1.98,1.63)
\psset{plotpoints=10000}
\pspolygon[linewidth=0.4pt,linecolor=black,fillstyle=solid,fillcolor=lightblue](1,0)(0,1)(-1,0)(0,-1)
\psaxes[ticks=none,labels=none,linewidth=0.6pt]{->}(0,0)(-1.6,-1.4)(1.7,1.5)%
[$\!\!\mu_1$,0][$ $,90]
\uput[0](-0.18,1.6){\small $\mu_2$}
\uput[0](-0.27,1.03){\small $1$}
\uput[0](-0.40,-1.12){\small $-1$}
\uput[0](0.9,0.12){\small $1$}
\uput[0](-1.42,0.12){\small $-1$}
\pscircle*(0,1){0.038}
\pscircle*(0,-1){0.038}
\pscircle*(1,0){0.038}
\pscircle*(-1,0){0.038}
\end{pspicture}
\quad
\begin{pspicture}(-1.9,-1.43)(2.03,1.63)
\psset{plotpoints=10000}
\psellipse[linewidth=0.4pt,linecolor=black,fillstyle=solid,fillcolor=lightblue](0,-0.5)(1.41421356237,0.5)
\pspolygon[linestyle=none,linecolor=black,fillstyle=solid,fillcolor=lightblue](0,1)(-1.3333333333,-0.3333333333)(1.3333333333,-0.3333333333)
\psaxes[ticks=none,labels=none,linewidth=0.6pt]{->}(0,0)(-1.9,-1.4)(1.9,1.5)%
[$\!\!\mu_1$,0][$ $,90]
\psline[linewidth=0.4pt](0,1)(1.3333333333,-0.3333333333)
\psline[linewidth=0.4pt](0,1)(-1.3333333333,-0.3333333333)
\pscircle*(0,1){0.038}
\pscircle*(0,-1){0.038}
\pscircle*(1,0){0.038}
\pscircle*(-1,0){0.038}
\uput[0](-0.18,1.6){\small $\mu_2$}
\uput[0](-0.27,1.03){\small $1$}
\uput[0](-0.40,-1.12){\small $-1$}
\uput[0](0.9,0.12){\small $1$}
\uput[0](-1.42,0.12){\small $-1$}
\uput[0](0.52,-0.35){\small $S$}
\end{pspicture}
\caption{Illustration of the conditions \eqref{NO-multiplier} and \eqref{pos-prop},
left and right, respectively, for $\mu_3=\dotsb=\mu_6=0$; cf.\ \eqref{p-discrep4}.}
\label{Fi:NO-posit}
\end{figure}
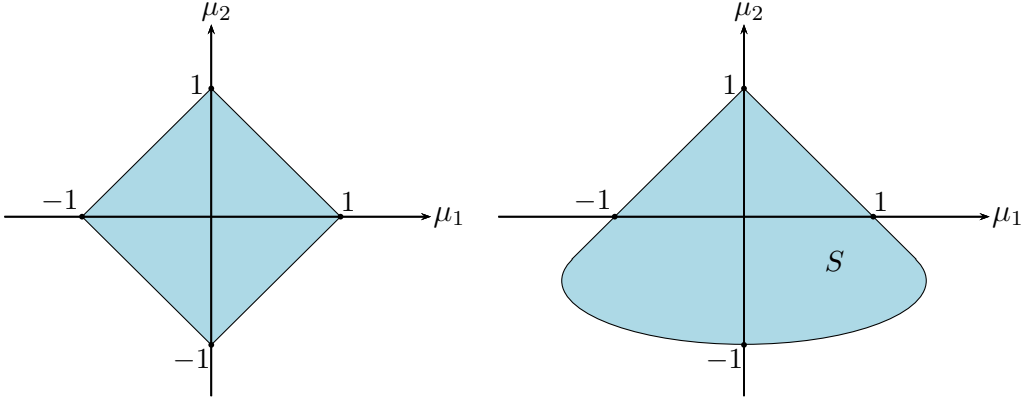

\subsection{On the construction of multipliers}\label{SSe:construction}
In this part, we describe some necessary conditions of multipliers satisfying
the A-stability condition \eqref{A} and the relaxed positive condition \eqref{pos-prop}.
To begin with, we show that no multiplier with $\mu_3=\dotsb=\mu_6=0$ exists.

\begin{proposition}\label{prop:no-2m}
There is no multiplier with $\mu_3=\dotsb=\mu_6=0$, satisfying  \eqref{A}
and \eqref{pos-prop}.
\end{proposition}

\begin{proof}
The positivity condition  \eqref{pos-prop} is satisfied if and only if
\begin{equation}
\label{multiplier-2}
1-\mu_1x-\mu_2(2x^2-1)>0\quad \forall x\in [-1,1];
\end{equation}
see \eqref{p-discrep1}.
The A-stability condition \eqref{A} is in this case equivalent to
\eqref{Real5} with
\begin{equation}
\label{multiplier-4}
\begin{aligned}
P(x)
&=(8+59x-157x^2)\big (4x^3-\mu_1(2x^2-1)-3x-\mu_2x\big )\\
&\quad+(1+x)(137x^2-144x+22)(4x^2-2\mu_1x-\mu_2-1).
\end{aligned}
\end{equation}
%


First,  the estimate $|\mu_1|<\sqrt{2}$ in \eqref{eqn:mu1-mu2} and  the nonnegativity of
\[P(- 4/{25})=-41.65312\mu_2+7.86979\mu_1-39.13478\]
lead to
\begin{equation}
\label{multiplier-4n}
\mu_2<
\frac {7.86979\sqrt{2}-39.13478}{41.65312}=-0.672343782385853.
\end{equation}
%

On the other hand, for $\mu_2<-0.672343782385853,$ we have $|\mu_2|>|\mu_1|/4,$ and thus $(\mu_1,\mu_2)$ must lie
in the interior of the ellipse in \eqref{p-discrep4}. Now, $P(0.99)=a\mu_2+b\mu_1+c$ with
\[a=\frac {2086460708677967}{35184372088832},\quad b=\frac{1053766469372221}{35184372088832},\quad c=\frac{9685378027}{109951162777600},\]
and the intersection points of the line $P(0.99)=0$ and the ellipse $4 (\mu_2+1/2 )^2+ \mu_1^2/2 =1$ are
\begin{equation}
\label{multiplier-AB}
\left\{\begin{aligned}
&A=(2.941186035762484\cdot 10^{-6}, -1.08131109678632\cdot 10^{-12}),\\
&B=(1.328818676149621, -0.671118740185537).
\end{aligned}
\right.
\end{equation}
It is easily seen that  $P(0.99)$ is nonnegative only in the part of the interior of the ellipse
to the right of the segment $AB;$ cf.\ Figure \ref{Fi:mult-2}.
Therefore, $P(0.99)\geqslant 0$ implies
\[\mu_2\geqslant -0.671118740185537.\]
This together with  \eqref{multiplier-4n} leads to a contradiction;
hence, no multiplier of
the form $(\mu_1,\mu_2,0,$ $\dotsc,0)$ exists for the six-step BDF method.
\end{proof}


\begin{figure}[!ht]
\centering
\psset{yunit=1.7cm,xunit=1.7}
\begin{pspicture}(-2,-1.5)(2,0.6)
\pscustom[linewidth=0.01pt,fillstyle=solid,fillcolor=lightblue]{
\pscurve[linewidth=0.01pt](1.328818676149621,-0.671118740185535)(1.41421356237,-0.5)
(1.3333333333,-0.3333333333)(1.1,-0.185754872750587)(1,-0.146446609406726)
(0.9,-0.114318784486462)(0.6,-0.047230743093129)(0.3,-0.011379492857698)
(0.000002941186035762484,-0.000000000001081311309678635)
\psline[linewidth=0.0pt](0.000002941186035762484,-0.000000000001081311309678635)(1.328818676149621,-0.671118740185535)}
\psaxes[ticks=none,labels=none,linewidth=0.6pt]{->}(0,0)(-1.9,-1.4)(1.9,0.5)%
[$\!\!\mu_1$,0][$ $,90]
\psset{plotpoints=10000}
\psellipse[linewidth=0.4pt,linecolor=black,fillcolor=lightblue](0,-0.5)(1.41421356237,0.5)
\pscircle*(0,-1){0.038}
\pscircle*(0,-0.5){0.038}
\pscircle*(1,0){0.038}
\pscircle*(-1,0){0.038}
\pscircle*(1.328818676149621,-0.671118740185535){0.038}
\pscircle*(0.000002941186035762484,-0.000000000001081311309678635){0.038}
\psline[linewidth=0.4pt](0.000002941186035762484,-0.000000000001081311309678635)(1.328818676149621,-0.671118740185535)
\uput[0](0.83,0.12){\small $1$}
\uput[0](-1.35,0.12){\small $-1$}
\uput[0](-0.435,-0.50){\small $-\frac 12$}
\uput[0](-0.40,-1.12){\small $-1$}
\uput[0](1.22,-0.78){\small $B$}
\uput[0](-0.3,0.13){\small $O$}
\uput[0](-0.05,0.13){\small $A$}
\uput[0](-0.18,0.6){\small $\mu_2$}
\end{pspicture}
\caption{Out of the interior points $(\mu_1,\mu_2)$ of the ellipse, $P(0.99)$, see \eqref{multiplier-4}, is nonnegative only
in the blue region; in the blue region, $\mu_2\geqslant -0.671118740185535$. The points $A$ and $B$
are given in \eqref{multiplier-AB}. The discrepancy between $A$ and $O=(0,0)$ is invisible.}
\label{Fi:mult-2}
\end{figure}
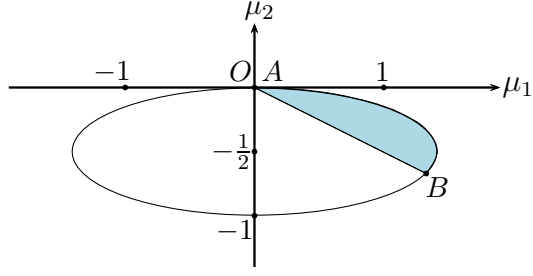

Our next attempt was to seek a multiplier with $\mu_4=\mu_5=\mu_6=0$.
In this case, the A-stability condition \eqref{A} and the positivity condition  \eqref{pos-prop}
lead, respectively, to the conditions
\begin{equation}
\label{multiplier-7}
\begin{aligned}
P(x)
&=(8+59x-157x^2)\big (4x^3-\mu_1(2x^2-1)-3x-\mu_2x-\mu_3\big )\\
&\quad+(1+x)(137x^2-144x+22)(4x^2-2\mu_1x-\mu_2-1)\geqslant 0 \quad \forall x\in [-1,1]
\end{aligned}
\end{equation}
and
\begin{equation}
\label{multiplier-8}
f(x):=1-\mu_1\cos x-\mu_2\cos (2x)-\mu_3\cos (3x) >0 \quad \forall x \in \Re.
\end{equation}

Necessary conditions for \eqref{multiplier-7} and \eqref{multiplier-8} could be derived
by evaluating $P$ and $f$ at certain points.
For instance, we claim the following necessary condition,
which helps us to construct multipliers.

\begin{proposition}\label{prop:nec-cond}
If $(\mu_1,\mu_2,\mu_3,0,0,0)$ is a multiplier of the six-step BDF method, then there holds
\begin{equation*}
0.41990729<\mu_1<\sqrt{3}, \  -1<\mu_2< -0.58852878,\ 0<\mu_3<1, \ |\mu_1|+|\mu_2|+|\mu_3|>1.
\end{equation*}
\end{proposition}

\begin{proof}
First, $|\mu_2|<1$ follows immediately from the positivity of
$f(\pi/2)$ and of $f(0)$ and $f(\pi).$  Furthermore,
\[2f(2\pi/3)+f(0)=3(1-\mu_3)\quad\text{and}\quad 2f(\pi/3)+f(\pi)=3(1+\mu_3),\]
whence $|\mu_3| <1.$ In view of
\[f(\pi/6)=\frac{1}{2}\big(-\sqrt{3}\mu_1-\mu_2+2\big)\quad\text{and}\quad f(5\pi/6)=\frac{1}{2}\big(\sqrt{3}\mu_1-\mu_2+2\big),\]
we have $\sqrt{3}|\mu_1|<2-\mu_2,$ and, in combination with $\mu_2>-1,$ infer that
$|\mu_1|<\sqrt{3}.$

Up to this point, we did not use the nonnegativity of $P$.  Now we check $P(0)\geqslant 0$, i.e.,
\begin{equation}
\label{multiplier-10}
P(0)=2\big [4(\mu_1-\mu_3)-11(1+\mu_2)\big ] \geqslant  0.
\end{equation}
Since $1+\mu_2>0,$ we infer that $\mu_3<\mu_1$. Furthermore, since $\mu_1<\sqrt{3}$ and  $|\mu_3| <1,$
\[11\mu_2< 4\big (\sqrt{3}+1\big ) -11<-0.07179,\quad\text{whence}\quad \mu_2<-0.65263636\cdot 10^{-2}.\]
Meanwhile, since $274/625+1154\mu_2/25<0,$ the nonegativity of
\begin{equation}
P(0.8)=\frac {274}{625}+\frac {1154}{25}\mu_2+\frac {3572}{125}\mu_1+\frac {1132}{25}\mu_3
\end{equation}
yields $3572 \mu_1/125+1132\mu_3/25> 0$, which together with $\mu_3<\mu_1$
 leads to
\[\frac {3572}{125}\mu_1+\frac {1132}{25}\mu_1>\frac {3572}{125}\mu_1+\frac {1132}{25}\mu_3> 0,\]
i.e., $\mu_1>0.$ Therefore, we arrive at
\begin{equation*}
0<\mu_1<\sqrt{3},\quad -1<\mu_2<-0.65263636\cdot 10^{-2}\quad\text{and}\quad0<|\mu_3|<1.
\end{equation*}

Next, we prove $\mu_3>0$ by contradiction. If  $\mu_3\leqslant 0$, then the positivity of $f(\pi/4)$ yields
\[ f(\pi/4)=1- \frac{\sqrt{2}}{2  } (\mu_1-\mu_3)>0
 \implies\mu_1<\sqrt{2}.\]
This and the nonnegativity of $P(-4/25)$ imply $\mu_2<-0.672$. Then, we can derive a lower bound $\mu_1>1.3426$
by examining $P(0.999)\geqslant  0$.
However, with $\mu_1>1.3426$, $\mu_2<-0.672$ and $\mu_3\leqslant 0$, it is easy to observe that
\begin{equation*}
 2f(\pi/3)=-\mu_1+\mu_2+2\mu_3+2\leqslant - 1.3426-0.672+2<   -0.0146,
\end{equation*}
which violates the positive condition \eqref{multiplier-8}. Therefore, we conclude that $\mu_3>0$.

Moreover, from $\mu_1<\sqrt3, \mu_3>0$ and   the nonnegativity of
\[P(- 66/625)=7.33518936\mu_1 -34.64182239\mu_2  -0.01883648\mu_3-33.09263039,\]
we infer that
\begin{equation*}
\label{multiplier-4n}
\mu_2<
\frac {7.33518936\sqrt{3}-33.09263039}{34.64182239}<-0.58852878.
\end{equation*}
Then, the  nonnegativity of $P(27/125)$ yields $ \mu_1> 0.41990729$. Thus, we arrive at
\begin{equation*}
0.41990729<\mu_1<\sqrt{3},\quad-1<\mu_2<-0.58852878\quad \text{and}\quad 0<\mu_3<1.
\end{equation*}

Finally, the property $|\mu_1|+|\mu_2|+|\mu_3|>1$ is a special case
of the more general result of the next Remark.
\end{proof}


\begin{remark}[Nonexistence of Nevanlinna--Odeh multipliers for the six-step BDF method]\label{rem:mult-general}
The multiplier \eqref{mu2} is not unique. In general, the A-stability condition \eqref{A}
and the positivity condition  \eqref{pos-prop}  lead to the conditions
\begin{equation*}
\begin{aligned}
  P(x)
  ={}&(-80x^5+208x^4-122x^3-82x^2+98x-22)+(40x^4-104x^3+71x^2+15x+8)\mu_1 \\
  &+(20x^3-52x^2+114x-22)\mu_2-(8+59x-157x^2)\mu_3\\
    & +(294x^3-66x^2-130x+22)\mu_4+(588x^4-132x^3-417x^2+103x+8)\mu_5\\
    &+(1176x^5-264x^4-1128x^3+272x^2+146x-22)\mu_6\geqslant 0
\end{aligned}
\end{equation*}
and
\begin{equation*}
\begin{aligned}
p(x)
={}&1-x\mu_1-(2x^2-1)\mu_2-(4x^3-3x)\mu_3-(8x^4-8x^2+1)\mu_4\\
&-(  16x^5-20x^3+5x     )\mu_5-(   32x^6-48x^4+18x^2-1   )\mu_6>0,
\end{aligned}
\end{equation*}
respectively, for all $x\in [-1,1]$. In Table \ref{Ta:mult-four},
we list several multipliers satisfying these conditions. 

Furthermore, evaluating $P$ at $x=3/40,$ we have
\[P\big (3/40\big )< -15.1563+13.7341 \sum_{i=1}^6|\mu_i|.\]
Assuming $|\mu_1|+\dotsb +|\mu_6|\leqslant 1,$ we observe that
\[P\big (3/40\big )< -1.4222 < 0,\]
and infer that no Nevanlinna--Odeh multiplier exists for the six-step BDF method.
\end{remark}

\begin{table}[!ht]
\begin{center}
\caption{Multipliers for the six-step BDF method;  
see also \eqref{mu2}.}
\begin{tabular}{|c|c|c|c|c|c|}
\hline
$\mu_1$ & $\mu_2$ & $\mu_3$  & $\mu_4$  & $\mu_5$ & $\mu_6$\\
\hhline{|======|}
$1.6$ & $-0.92$ & $0.3$   & $0$ & $0$ & $0$ \\
\hhline{|------|}
$0.8235$ & $-0.855$ & $0.38$  & $0$ & $0$ & $0$  \\
\hhline{|------|}
$1.67$ & $-1$ & $0.4$  & $-0.1$ & $0$ & $0$  \\
\hhline{|------|}
$0.8$ & $-0.7$ & $0.2$  & $0.1$ & $0$ & $0$  \\
\hhline{|------|}
$1.118$ & $-1$ & $0.6$  & $-0.2$ & $0.2$ & $0$  \\
\hhline{|------|}
$0.6708$ & $-0.2$ & $-0.2$  & $0.6$ & $-0.2$ & $0$  \\
\hhline{|------|}
$0.735$ & $-0.2$ & $-0.4$  & $0.8$ & $-0.4$ & $0.2$  \\
\hhline{|------|}
\end{tabular}
\label{Ta:mult-four}
\end{center}
\end{table}

\section{Stability}\label{Se:stab}
In this section we prove stability of the six-step BDF method \eqref{ab} by the energy
technique. The result is well known; the novelty is in the simplicity of the proof,
the  main advantage of the energy technique.
Proofs by other stability techniques are significantly more involved.
For a proof by a spectral technique in the case of selfadjoint operators, we refer to \cite[chapter 10]{T};
for a proof in the general case, under a sharp condition on the nonselfadjointness of the operator
as well as for nonlinear parabolic equations,
by a combination of spectral and Fourier techniques, see, e.g., \cite{A3} and references therein.
For a long-time estimate in the case of selfadjoint operators and an application to the Stokes--Darcy problem, see \cite{LiWangZhou:2020}.

For simplicity, we denote by $\langle\cdot,\cdot\rangle$ the inner product on $V,$ $\langle v, w\rangle:=(A^{1/2}v, A^{1/2}w).$

Before we proceed, for the reader's convenience, we recall the notion of the generating function of
an $n\times n$ Toeplitz  matrix $T_n$ as well as an auxiliary result, the Grenander--Szeg\"o theorem.

\begin{definition}[{\cite[p.\ 13]{Chan:07}}; the generating function of a Toeplitz matrix]\label{De:gen-funct}
Consider the $n \times n$ Toeplitz  matrix  $T_n=(t_{ij})\in \C^{n,n}$ with diagonal entries $t_0,$ subdiagonal entries
$t_1,$ superdiagonal entries $t_{-1},$ and so on, and $(n,1)$ and $(1,n)$ entries
$t_{n-1}$ and   $t_{1-n}$, respectively, i.e., the entries $t_{ij}=t_{i-j}, i,j=1,\dotsc,n,$ are constant along the diagonals of $T_n.$
 Let   $t_{-n+1},\dotsc, t_{n-1}$ be the Fourier coefficients of the trigonometric polynomial $f$, i.e.,
\begin{equation*}
  t_k=\frac{1}{2\pi}\int_{-\pi}^{\pi}f(x)\e^{-\i kx}\, \mathrm{d} x,\quad k=1-n,\dotsc,n-1.
\end{equation*}
Then, $f, f(x)=\sum_{k=1-n}^{n-1} t_k\e^{\i kx},$ is called  \emph{generating function} of $T_n$.
\end{definition}

If the generating function $f$ is real-valued, then the matrix $T_n$ is Hermitian; if $f$ is real-valued and even, then  $T_n$ is symmetric.

Notice, in particular, that the generating function of a symmetric band Toeplitz matrix of bandwidth $2m+1,$
i.e., with $t_{m+1}=\dotsb=t_{n-1}=0,$ is a real-valued, even trigonometric polynomial, $f(x)= t_0+2t_1 \cos x+\dotsb+2t_m \cos (mx),$
for all $n\geqslant m+1.$

\begin{lemma}[{\cite[pp.\ 13--14]{Chan:07}}; the Grenander-Szeg\"{o} theorem]\label{Le:GS}
Let $T_n$ be a symmetric Toeplitz  matrix as in Definition \ref{De:gen-funct} with generating function $f$.
Then, the smallest and largest eigenvalues  $\lambda_{\min}(T_n)$ and $\lambda_{\max}(T_n)$, respectively, of $T_n$
are bounded as follows
\begin{equation*}
  f_{\min} \leqslant \lambda_{\min}(T_n) \leqslant \lambda_{\max}(T_n) \leqslant f_{\max},
\end{equation*}
with $f_{\min}$ and  $f_{\max}$  the minimum and maximum of $f$, respectively.
In particular, if  $f_{\min}$ is positive, then the symmetric matrix $T_n$ is positive
definite.\footnote{For real-valued $f$ and $z=(z_0,\dotsc,z_{n-1})^\top\in \C^n,$ we have
$(T_nz,z)=\frac{1}{2\pi}\displaystyle{\int_{-\pi}^{\pi}f(x)\Big |\sum_{k=0}^{n-1} z_k\e^{\i kx}\Big |^2\, \mathrm{d} x}$
and $(z,z)=\frac{1}{2\pi}\displaystyle{\int_{-\pi}^{\pi}\Big |\sum_{k=0}^{n-1} z_k\e^{\i kx}\Big |^2\, \mathrm{d} x}$,
and the result is evident.}
\end{lemma}

\begin{theorem}[Stability of the six-step BDF method]\label{Th:stab}
The six-step BDF method \eqref{ab} is stable in the sense that 
\begin{equation}
\label{stab-abg2}
|u^n|^2 + \tau \sum_{\ell=6}^n \|u^\ell\|^2
\leqslant C\sum_{j=0}^5 \big (|u^j|^2+\tau \|u^j\|^2\big ),\quad n=6,\dotsc,N,
\end{equation}
%
with a  constant $C$  independent of $\tau$ and $n$.
\end{theorem}

\begin{proof}
 Taking in \eqref{ab} the inner product with
$u^{n+6}-\frac {13}9u^{n+5}+\frac {25}{36}u^{n+4}-\frac 19u^{n+3}$, cf.\ \eqref{ab-energy}, we have
 \begin{equation}
\label{abg31}
\Big (\sum\limits^6_{i=0}\alpha_i  u^{n+i},u^{n+6}-\sum_{j=1}^3\mu_j u^{n+6-j}\Big )+\tau I_{n+6}=0
\end{equation}
with
 \begin{equation}\label{2.a7}
I_{n+6}:=\Big \langle u^{n+6},u^{n+6}-\sum_{j=1}^3\mu_j u^{n+6-j}\Big \rangle.
\end{equation}

With the notation $\mathcal{U}^n:=(u^{n-5},u^{n-4},u^{n-3},u^{n-2},u^{n-1},u^{n})^\top$ and the norm $|\mathcal{U}^n|_G$ given by
 \begin{equation*}
 |\mathcal{U}^n|_G^2=\sum_{i,j=1}^6g_{ij}\left(u^{n-6+i},u^{n-6+j}\right),
\end{equation*}
using \eqref{G}, we have
\begin{equation}
\Big (\sum\limits^6_{i=0}\alpha_i  u^{n+i},u^{n+6}-\sum_{j=1}^3\mu_j u^{n+6-j}\Big )\geqslant |\mathcal{U}^{n+6}|_G^2-|\mathcal{U}^{n+5}|_G^2.
\end{equation}

Thus, \eqref{abg31} yields
\begin{equation}\label{2.87}
 |\mathcal{U}^{n+6}|_G^2-|\mathcal{U}^{n+5}|_G^2+\tau I_{n+6}\leqslant 0.
\end{equation}
Summing in \eqref{2.87} from $n=0$ to $n=m-6$,  we obtain
\begin{equation}\label{2.88}
 |\mathcal{U}^{m}|_G^2-|\mathcal{U}^{5}|_G^2+\tau \sum_{n=6}^mI_n\leqslant 0.
\end{equation}
 It remains to estimate the sum $\sum_{n=6}^mI_n$ from below; we have
 \begin{equation}
\label{abg36}
\sum_{n=6}^mI_n=\sum_{n=6}^m\Big \langle u^n,u^n-\sum_{j=1}^3\mu_j u^{n-j}\Big \rangle.
\end{equation}
First, motivated by the positivity of the function $f$ of \eqref{f}, to take advantage of the
positivity property  \eqref{pos-prop}, we introduce the notation $\mu_0:=-31/32,$
and rewrite \eqref{abg36} as
 \begin{equation}
\label{abg37}
\sum_{n=6}^mI_n=\frac 1{32}\sum_{n=6}^m\|u^n\|^2+J_m\ \text{ with }\
J_m:=-\sum_{j=0}^3\mu_j \sum_{i=1}^{m-5}\langle u^{5+i},u^{5+i-j} \rangle.
\end{equation}

Our next task is to rewrite $J_m$ in a form that will enable us to estimate
it from bellow in a desired way.
To this end, we introduce the lower triangular Toeplitz matrix $L=(\ell_{ij})\in \Re^{m-5,m-5}$ with entries
\[\ell_{i,i-j}=-\mu_j, \quad j=0,1,2,3, \quad i=j+1,\dotsc,m-5,\]
and all other entries equal zero.
With this notation, we have
\[\sum_{i,j=1}^{m-5}\ell_{ij} \langle u^{5+i},u^{5+j} \rangle=-\sum_{j=0}^3\mu_j \sum_{i=j+1}^{m-5}\langle u^{5+i},u^{5+i-j} \rangle,\]
i.e.,
\begin{equation}
\label{abg38}
\sum_{i,j=1}^{m-5}\ell_{ij} \langle u^{5+i},u^{5+j} \rangle=J_m+\langle u^6,\mu_1u^5+\mu_2u^4+\mu_3u^3 \rangle
+\langle u^7,\mu_2u^5\!+\!\mu_3u^4 \rangle+\langle u^8,\mu_3u^5 \rangle.
\end{equation}
%

Now, in view of the positivity of the generating function $f,$ see \eqref{f}, of the symmetric part
$L_s:=(L+L^\top)/2$ of the matrix $L,$ the Grenander--Szeg\"o theorem, see Lemma \ref{Le:GS},
ensures positive definiteness of $L_s,$ and thus also of $L$ itself, since $(Lx,x)=(L_sx,x)$ for $x\in \Re^{m-5}.$
Therefore, the expression on the left-hand side of \eqref{abg38} is nonnegative; hence,
\eqref{abg38} yields the desired estimate for $J_m$ from below,
i.e.,
\begin{equation}
\label{abg39}
J_m\geqslant -\langle u^6,\mu_1u^5+\mu_2u^4+\mu_3u^3 \rangle
-\langle u^7,\mu_2u^5\!+\!\mu_3u^4 \rangle-\langle u^8,\mu_3u^5 \rangle.
\end{equation}

From \eqref{2.88}, \eqref{abg37}   and \eqref{abg39}, we obtain
\begin{equation}\label{abg40}
\begin{aligned}
 |\mathcal{U}^{m}|_G^2+\frac 1{32}\sum_{n=6}^m\|u^n\|^2&\leqslant
 |\mathcal{U}^{5}|_G^2 +\tau\langle u^6,\mu_1u^5+\mu_2u^4+\mu_3u^3 \rangle\\
&+\tau \langle u^7,\mu_2u^5+\mu_3u^4 \rangle+\tau \langle u^8,\mu_3u^5 \rangle.
\end{aligned}
 \end{equation}
Now, with $c_1$ and $c_2$ the smallest and largest eigenvalues of the matrix $G,$ we have
\[  |\mathcal{U}^{m}|_G^2\geqslant c_1|u^m|^2\quad\text{and}\quad |\mathcal{U}^{5}|_G^2\leqslant
c_2\sum_{j=0}^5 |u^j|^2;\]
furthermore, the terms $|\langle u^i,u^j \rangle|$ with $i>j$ can be estimated in the form
$|\langle u^i,u^j \rangle|\leqslant \varepsilon \|u^i\|^2+\|u^j\|^2/(4\varepsilon)$
with $\varepsilon<1/32.$  This leads then to the desired stability estimate
\eqref{stab-abg2}.

Let us also note that, due to the fact that $\mu_4=\mu_5=\mu_6=0,$
the terms $\|u^2\|^2, \|u^1\|^2$ and $\|u^0\|^2$ are actually not needed
on the right-hand side of \eqref{stab-abg2}.
\end{proof}

\subsection{Time-dependent operators}\label{SSe:t-dep}
In this section we use a perturbation argument to extend the stability result to the case
of time-dependent selfadjoint operators $A (t) : V \to V', t\in [0,T].$
We fix an $s\in [0,T]$ and define the norm on $V$ in terms of  $A (s), \|v\|:=|A (s)^{1/2}v|.$

Our structural assumptions are that all operators $A(t), t\in [0,T],$ share the same domain,
produce equivalent norms on $V,$
\[|A(t)^{1/2}v|\leqslant c |A(\tilde t)^{1/2}v|\quad  \forall t,\tilde t\in [0,T]\ \, \forall v\in V,\]
and $A(t): V\to V'$ is of  bounded variation with respect to $t,$
\begin{equation}
\label{BV-A}
 \|\big (A(t)- A(\tilde t)\big )v\|_\star \leqslant  [\sigma(t)-\sigma(\tilde t) ] \|v\|,\quad
 0\leqslant \tilde t\leqslant t\leqslant T, \quad \forall v\in V,
\end{equation}
with an increasing function $\sigma : [0,T]\to \Re.$ 

First, for given perturbation terms $v^6,\dotsc,v^N\in V',$ we let $u^6,\dotsc,u^N$
satisfy the perturbed six-step BDF method
\begin{equation}
\label{ab-v}
\sum_{i=0}^6 \alpha_i u^{n+i}+\tau A u^{n+6}=\tau v^{n+6},\quad n=0,\dotsc,N-6,
\end{equation}
i.e., the scheme \eqref{ab} for $q=6$ with perturbed right-hand side,
assuming that starting approximations $u^0, \dotsc, u^5$ are given.
Then, it is easily seen that we have the following stability result
\begin{equation}
\label{stab-v}
|u^n|^2 + \tau \sum_{\ell=6}^n \|u^\ell\|^2
\leqslant C\sum_{j=0}^5 \big (|u^j|^2+\tau \|u^j\|^2\big ) + C\tau \sum_{\ell=6}^n \|v^\ell\|^2_\star,\quad n=6,\dotsc,N,
\end{equation}
with a  constant $C$  independent of $\tau$ and $n$. Indeed, the terms that are due to the perturbation $v^{n+6},$
namely, $\big ( v^{n+6}, u^{n+6}-\sum_{j=1}^3\mu_j u^{n+6-j}\big )$,
can be easily estimated in the form
\[\big (  v^{n+6}, u^{n+6}-\sum_{j=1}^3\mu_j u^{n+6-j}\big ) \leqslant C_\varepsilon \|v^{n+6}\|^2_\star+\varepsilon \sum_{j=1}^3|\mu_j|\, \|u^{n+6-j}\|^2,\]
with sufficiently small $\varepsilon$ such that the terms involving $\|u^i\|^2$ can be absorbed in the corresponding
sum on the left-hand side.

We shall next use \eqref{stab-v} to extend the stability result \eqref{stab-abg2} to the case of
time-dependent operators. Before that, let us note that with  $v^\ell=f(t^\ell),$  \eqref{stab-v} is a stability
result for the inhomogeneous equation $u' (t) + Au(t) = f(t)$ with respect to both the starting approximations
and the forcing term. Furthermore, with $v^\ell$ the consistency error of the method,
i.e., the amount by which the exact solution $u$ misses satisfying the numerical method \eqref{ab},
with $q=6,$
\begin{equation}
\label{ab-v-cons}
\tau v^{n+6}=\sum_{i=0}^6 \alpha_i u(t^{n+i})+\tau A u(t^{n+6})-\tau f(t^{n+6}),\quad n=0,\dotsc,N-6,
\end{equation}
the error $e^\ell:=u(t^\ell)-u^\ell, \ell=0,\dotsc,N,$ satisfies \eqref{ab-v}. In this case, the stability
result \eqref{stab-v}, in combination with the trivial estimate of the consistency error,
leads to optimal order error estimates.

Now, the six-step BDF method for the initial value problem \eqref{ivp} with time-dependent
operator $A(t)$ is
\begin{equation}
\label{ab-time}
\sum_{i=0}^6 \alpha_i u^{n+i}+\tau A(t^{n+6}) u^{n+6}=0,\quad n=0,\dotsc,N-6,
\end{equation}
assuming that starting approximations $u^0, \dotsc, u^5$ are given.
Let us now fix an $6\leqslant m\leqslant N.$ From \eqref{ab-time}, we obtain
\begin{equation}
\label{ab-time2}
\sum_{i=0}^6 \alpha_i u^{n+i}+\tau A(t^m) u^{n+6}=\tau \big [A(t^m)-A(t^{n+6})\big ] u^{n+6},\quad n=0,\dotsc,m-6.
\end{equation}
Since the time $t$ is frozen at $t^m$ in the operator $A(t^m)$ on the left-hand side,
we can apply the already-established stability estimate \eqref{stab-v}
with  perturbation terms $v^\ell:= [A(t^m)-A(t^\ell) ] u^\ell$ and obtain
\begin{equation}
\label{ab-time3}
|u^m|^2 + \tau \sum_{\ell=6}^m \|u^\ell\|^2
\leqslant  C\sum_{j=0}^5\big (|u^j|^2+\tau \|u^j\|^2\big ) + C M^m
\end{equation}
 with a  constant $C$  independent of $\tau$ and $m$, and
\begin{equation}
\label{ab-time4}
M^m:=\tau\sum_{\ell=6}^m \| [A(t^m)-A(t^\ell) ] u^\ell\|^2_\star.
\end{equation}

Now, with
\[E^\ell:=\tau \sum_{j=6}^\ell \|u^j\|^2,\quad \ell=6,\dotsc,m,\quad  E^5:=0,\]
estimate \eqref{ab-time3} yields
\begin{equation}
\label{ab-time5}
E^m \leqslant C\sum_{j=0}^5 \big (|u^j|^2+\tau \|u^j\|^2\big ) + C M^m.
\end{equation}
Furthermore, in view of the bounded variation condition \eqref{BV-A},
\[M^m\leqslant \tau\sum\limits^{m-1}_{\ell=6}\big [\sigma(t^m)-\sigma(t^\ell)\big ]^2\|u^\ell\|^2
=\sum\limits^{m-1}_{\ell=6}\big [\sigma(t^m)-\sigma(t^\ell)\big ]^2(E^\ell-E^{\ell-1}),\]
whence, by summation be parts, we have
\begin{equation}
\label{ab-time6}
M^m\leqslant \sum\limits^{m-1}_{\ell=6}a_\ell E^\ell,
\end{equation}
with $a_\ell:=\big [\sigma(t^m)-\sigma(t^\ell)\big ]^2-\big [\sigma(t^m)-\sigma(t^{\ell+1})\big ]^2,$
and \eqref{ab-time5} yields
\begin{equation}
\label{ab-time7}
E^m \leqslant C\sum_{j=0}^5\big (|u^j|^2+\tau \|u^j\|^2\big )+C\sum\limits^{m-1}_{\ell=6}a_\ell E^\ell.
\end{equation}
Since the sum $\sum_{\ell=6}^{m-1} a_\ell$ is uniformly bounded by a constant independent
of $m$ and the time step $\tau,$
\[\sum_{\ell=6}^{m-1} a_\ell=\big [\sigma(t^m)-\sigma(t^6)\big ]^2\leqslant \big [\sigma(T)-\sigma(0)\big ]^2,\]
a discrete Gronwall-type argument  applied to \eqref{ab-time7}
 leads to
\begin{equation}
\label{ab-time8}
E^m\leqslant C\sum\limits^5_{j=0}   \big (|u^j|^2+\tau \|u^j\|^2\big ) .
\end{equation}
Combining \eqref{ab-time3} with \eqref{ab-time6} and \eqref{ab-time8},
we obtain the desired stability estimate  \eqref{stab-abg2} for the case
of time-dependent operators.



\bibliographystyle{amsplain}

\end{document}